\documentclass[10pt,reqno]{amsart}

\usepackage[a4paper,left=35mm,right=35mm,top=30mm,bottom=30mm,marginpar=25mm]{geometry}

\usepackage{amsmath}
\usepackage{amssymb}
\usepackage{amsthm}
\usepackage[latin1]{inputenc}
\usepackage{eurosym}
\usepackage[dvips]{graphics}
\usepackage{graphicx}
\usepackage{epsfig}
\usepackage{dsfont}
\usepackage{color}

\usepackage[displaymath,mathlines]{lineno}

\allowdisplaybreaks

\usepackage{ifthen}

\makeindex

\renewcommand{\div}{\operatorname{div}}

\newcommand{\Rr}{{\mathbb{R}}}

\newcommand{\Tt}{{\mathbb{T}}}

\newcommand{\Hh}{{\overline{H}}}

\def\leq{\leqslant}
\def\geq{\geqslant}

\newtheoremstyle{thmlemcorr}{10pt}{10pt}{\itshape}{}{\bfseries}{.}{10pt}{{\thmname{#1}\thmnumber{
#2}\thmnote{ (#3)}}}
\newtheoremstyle{thmlemcorr*}{10pt}{10pt}{\itshape}{}{\bfseries}{.}\newline{{\thmname{#1}\thmnumber{
#2}\thmnote{ (#3)}}}
\newtheoremstyle{defi}{10pt}{10pt}{\itshape}{}{\bfseries}{.}{10pt}{{\thmname{#1}\thmnumber{
#2}\thmnote{ (#3)}}}
\newtheoremstyle{remexample}{10pt}{10pt}{}{}{\bfseries}{.}{10pt}{{\thmname{#1}\thmnumber{
#2}\thmnote{ (#3)}}}
\newtheoremstyle{ass}{10pt}{10pt}{}{}{\bfseries}{.}{10pt}{{\thmname{#1}\thmnumber{
A#2}\thmnote{ (#3)}}}

\theoremstyle{thmlemcorr}

\theoremstyle{thmlemcorr*}
\newtheorem{theorem*}{Theorem}
\newtheorem{lemma*}[theorem]{Lemma}
\newtheorem{corollary*}[theorem]{Corollary}
\newtheorem{proposition*}[theorem]{Proposition}
\newtheorem{problem*}[theorem]{Problem}
\newtheorem{conjecture*}[theorem]{Conjecture}

\theoremstyle{defi}

\newtheorem{hyp}{A}

\theoremstyle{remexample}

\newtheorem{teo}{Theorem}
\newtheorem{Lemma}{Lemma}
\newtheorem{Proposition}{Proposition}
\newtheorem{Corollary}{Corolary}

\newtheorem{clm}{Claim}

\theoremstyle{ass}

\begin{document}
\title[Stationary mean-field games]{Regularity for second order stationary mean-field games}

\author{Edgard A. Pimentel}
\address[E. A. Pimentel]{
	Universidade Federal do Cear\'a, Campus of Pici, Bloco 914, Fortaleza, Cear\'a 60.455-760, Brazil.}
\email{epimentel@mat.ufc.br}
\author{Vardan Voskanyan}
\address[V. Voskanyan]{
	King Abdullah University of Science and Technology (KAUST), CEMSE Division , Thuwal 23955-6900. Saudi Arabia, and  
KAUST SRI, Center for Uncertainty Quantification in Computational Science and Engineering.}
\email{vardan.voskanyan@kaust.edu.sa}


\thanks{E. Pimentel was supported by CNPq-Brazil, grant $\#$167556/2014-2.}
\date{\today}

\begin{abstract}
In this paper, we prove the existence of classical solutions for second order stationary mean-field game systems. These arise in ergodic (mean-field) optimal control, convex degenerate problems in calculus of variations, and in the study of long-time behavior of time-dependent mean-field games. Our argument is based on the interplay between the regularity of solutions of the Hamilton-Jacobi equation in terms of the solutions of the Fokker-Planck equation and vice-versa. Because we consider different classes of couplings, distinct techniques are used to obtain a priori estimates for the density. In the case of polynomial couplings, we recur to an iterative method. An integral method builds upon the properties of the logarithmic function in the setting of logarithmic nonlinearities. This work extends substantially previous results by allowing for more general classes of Hamiltonians and mean-field assumptions.
\end{abstract}

\maketitle
\tableofcontents

\section{Introduction}

In this paper, we prove the existence of $C^{\infty}$ solutions for the following stationary mean-field game problem:

\begin{equation}\label{mfg}
\begin{cases}
\Delta u(x)\,+\,H(x,Du(x))\,=\,g(m(x))\,+\,\overline{H},&\,x\in\Tt^d\\
\Delta m(x)\,-\div\left(D_pH(x,Du)m(x)\right)\,=\,0,&\,x\in\Tt^d.
\end{cases}
\end{equation} 
Here, the unknowns are the functions $u\colon \Tt^d\to \Rr$, $m\colon \Tt^d\to \Rr^+,$ and a constant $\overline H\in\Rr.$ We also assume that the function $m$ is a probability density, i.e. $m\geq 0$ and $\int_{\Tt^d}mdx=1$, and $u$ satisfies the normalizing condition $\int_{\Tt^d}udx=0.$

The system \eqref{mfg} was introduced and first investigated by J.-M. Lasry and P.-L. Lions in \cite{ll1,ll3,LCDF}. In \cite{ll1,ll3}, the authors established the existence of weak solutions for \eqref{mfg}. 

The model problem \eqref{mfg} arises in ergodic (mean-field) optimal control problems, see \cite{bensoussan}. For developments related to ergodic mean-field games, see \cite{bardierg} and \cite{feleqi}. An additional motivation for \eqref{mfg} is the long-time behavior of time-dependent mean-field games. For results in this direction, we refer the reader to \cite{CLLP} and \cite{cllp13}. Finally, we mention the connection between \eqref{mfg} and the theory of convex degenerate problems in calculus of variations. In this sense, our results most likely add to the regularity theory of those problems as well.

In \cite{GPatVrt}, the existence of smooth solutions for stationary mean-field games is established under quadratic type assumptions on $H$ and polynomial or logarithmic assumptions on $g$. The existence of classical solutions for related systems is addressed in \cite{GM} and \cite{LIONS}, whereas several a priori estimates are obtained in \cite{GPM1}. Stationary mean-field game problems with congestion effects have been considered in \cite{GMit}, in the setting of purely quadratic Hamiltonians. 

Recent developments on stationary mean-field game problems were reported in \cite{cirant}, where the integral Bernstein estimates were used in the context of mean-field games for the first time (see also \cite{Lions85}). In that paper, the author established the existence of weak solutions for \eqref{mfg} in bounded domains, in the presence of Neumann boundary conditions. The arguments in \cite{cirant} require minimal assumptions on the Hamiltonian and considers couplings bounded from below, growing at most polynomially. In \cite{bardifeleqi}, the authors consider stationary mean-field game systems in the presence of uniformly elliptic operators and nonlocal couplings. The existence of solutions $(u,m,\bar{H})\in\mathcal{C}^{2,\alpha}(\Tt^d)\times W^{1,p}(\Tt^d)\times\Rr$ is established under a set of assumptions on the Hamiltonian. 

Time-dependent mean-field games of second order have also been considered in the literature. For the seminal analysis of this problem, we refer the reader to \cite{ll2} and \cite{ll3}. Well-posedness in the class of weak solutions have been recently investigated in \cite{porweak} and \cite{cgbt}. See also \cite{porretta}. Classical solutions have been considered in \cite{GPM2}, \cite{GPM3} and \cite{GPim2}.

In the present article we prove the existence of classical solutions, i.e., of class $C^\infty$, for \eqref{mfg} under certain conditions on the Hamiltonian $H$ and the nonlinearity $g$. Our main Theorem reads as follows:

\begin{teo}\label{maintheorem}
Let the Assumptions A\ref{Hsub}-A\ref{gpos} (cf. Section \ref{assp}) hold. Assume that either Assumptions A\ref{ggrow}.a and A\ref{alpha} or A\ref{ggrow}.b and A\ref{gamma} (cf. Section \ref{assp}) are satisfied. Then, there exists a unique $C^{\infty}$ solution $(u,m,\overline{H})$ of \eqref{mfg}.
\end{teo}

Assumptions A\ref{Hsub}-A\ref{a2} refer to conditions imposed on the growth of Hamiltonian $H$. These are now standard type of assumptions in the literature of mean-field games. A model Hamiltonian satisfying those is the following: 
\[
H(x,p)\,=\,a(x)\left(1+|p|^2\right)^\frac{\gamma}{2}\,+\,V(x),
\]
where $a,\,V\in C^{\infty}(\Tt^d)$ with $a> 0$. Because of a duality argument, the exponent $\gamma$ determines the growth of the Lagrangian associated with the underlying ergodic optimal control problem. 

Assumptions A\ref{gpos}-A\ref{alpha} concern the nonlinearity $g$. In this paper, we consider both polynomial couplings ($g(m)=m^\alpha$, A\ref{ggrow}.a) as well as logarithmic nonlinearities ($g(m)=\ln(m)$, A\ref{ggrow}.b). These are general examples that present distinct difficulties. When considering the former coupling, the central issue regards the growth of $g$. In the latter case, two mathematical challenges must be addressed. First, unlike in the polynomial setting, bounds for $L^p$ norms of the coupling $g$ do not follow from estimates for $m$ in Lebesgue spaces. In addition, because $\ln(m)$ in not bounded from below, bounds for solutions of the Hamilton-Jacobi equation cannot be inferred from the optimal control formulation of the problem.

Our primary contribution in Theorem \ref{maintheorem} is a substantial improvement of the results in the literature by considering general Hamiltonians as well as less restrictive growth regimes for the coupling $g$. For polynomial couplings, our result improves those in \cite{GPatVrt} by allowing for a more general growth regime of $g$, in higher dimensions. It also applies to a larger class of polynomial couplings than in \cite{cirant}. Furthermore, Theorem \ref{maintheorem} accommodates logarithmic couplings, which are nonlinearities unbounded from below.
Moreover, when compared with the time-dependent case, our results yield more general conditions, as one should expect, cf. \cite{GPM3}, \cite{GPim2}.

In broad lines, the proof of Theorem \ref{maintheorem} relies on two critical steps. The first one explores the interplay between two types of a priori estimates: suitable norms of $g$ controlled in terms of norms of $Du$ and vice-versa. In the case of the Hamilton-Jacobi (HJ) equation, this is done by recurring to the integral Bernstein method. This yields Lipschitz regularity for the solutions of the (HJ) in terms of norms of $g$ in appropriate Lebesgue spaces. To investigate the integrability of the nonlinearity, we establish a priori bounds for suitable norms of $m$ and $\ln(m)$. The former case is addressed by means of an iterative argument. To tackle the latter case, we use an integral argument combined with Sobolev's Theorem and concavity properties of the logarithmic function. 

Once those bounds are obtained, we combine them to get a certain amount of a priori regularity. This allows us to apply the continuation method, as in \cite{GPatVrt} and \cite{GMit}, and conclude the proof. We observe that, although the problem \eqref{mfg} has a variational structure, our techniques are analytic in nature; hence, we believe they still apply when small perturbations of this system are considered.

The remainder of this paper is structured as follows: in Section \ref{maop} we put forward the main assumptions under which we work in the paper, along with an outline of the proof of Theorem \ref{maintheorem}. In the Section \ref{rfpp} we investigate the integrability of the Fokker-Planck equation. The log-integrability of solutions to the Fokker-Planck equation is the content of the Section \ref{intlog}. Section \ref{bernstein} presents a priori estimates for the Hamilton-Jacobi equation whereas the Section \ref{proofmt} details the proofs of Theorem \ref{teo1} and Theorem \ref{teo2}. The application of the continuation method is described in the Appendix \ref{exist}.
\vspace{.15in}

{\bfseries Acknowledgments:} The authors are grateful to M. Cirant and D. Gomes for useful comments and suggestions during the preparation of this paper.

\section{Main assumptions and outline of the paper}\label{maop}

\subsection{Main assumptions}\label{assp}

In what follows, we detail the assumptions under which we work throughout the paper. We start by presenting the assumptions regarding the Hamiltonian $H$.

\begin{hyp}\label{Hsub}
The Hamiltonian $H:\Tt^d\times\Rr^d\to\Rr$ is smooth in both arguments. Furthermore, 
\begin{enumerate}
\item For every $x\in\Tt^d$, $H(x,p)$ is strictly convex with respect to $p$, i.e., there exists a constant $\delta>0$ so that $$D^2_{pp}H(x,p)\,>0.$$
\item There exists a constant $C>0$ such that the Hamiltonian $H$ satisfies a growth condition of the form $$C|p|^{\gamma}-C\,\,\leq\,H(x,p)\,\leq\,C\,+\,C\left|p\right|^\gamma,$$for $1<\gamma$.
\end{enumerate}
\end{hyp}

\begin{hyp}\label{DpH}
There exists a constant $C>0$ such that  $$\left|D_pH(x,p)\right|\leq C+C\left|p\right|^{\gamma-1}.$$
\end{hyp}

\begin{hyp}\label{Lsup}
There exists $C>0$ such that
$$p\cdot D_pH(x,p)-H(x,p)\geq -C+C\left|p\right|^\gamma.$$
\end{hyp}

\begin{hyp}\label{a2}
There exists a constant $C>0$ so that $$\left|D_{x}H(x,p)\right|,\,\left|D_{xx}H(x,p)\right|\,\leq\,CH\,+\,C,$$ $$D^2_{xp}H\leq C|p|^{\gamma-1}+C,$$ and $$\left|D_{pp}H(x,p)\right|\,\leq\,C\left|p\right|^{\gamma-2}+C.$$
\end{hyp}

A typical Hamiltonian satisfying the former set of assumptions is given by $$H(x,p)\,=\,a(x)\left(1+\left|p\right|^2\right)^\frac{\gamma}{2}+V(x),$$where $a,\,V\in\mathcal{C}^\infty(\Tt^d)$ with $a(x)>0$.

The next set of assumptions concerns the nonlinearity $g$.

\begin{hyp}\label{gpos}
The nonlinearity $g:\Rr^+\to\Rr$ is smooth and increasing, with $g(1)\geq 0$.
\end{hyp}

\begin{hyp}[Growth conditions on $g$]\label{ggrow}

We consider two types of nonlinearities:

\begin{itemize}
\item[a.]
Power-like nonlinearity: there exists $C>0$ such that
\begin{equation*}
g[m](x)\leq
\begin{cases}
C m(x),&\;  m\leq 1\\
C m^\alpha(x),&\; m> 1,
\end{cases}
\end{equation*}for some $\alpha>0$.
\item[b.] 
Logarithmic nonlinearity: $g[m](x)=\ln m(x)$.
\end{itemize}
\end{hyp}

\begin{hyp}\label{alpha}
The exponent $\alpha$ in A\ref{ggrow} is such that

$$\alpha\,\leq\,\frac{\gamma}{d(\gamma-1)}.$$
\end{hyp} 

\begin{hyp}\label{gamma}
The exponent $\gamma$ is so that $1\,<\,\gamma\,<\,2+\frac{1}{d-1}$.
\end{hyp}

\subsection{Outline of the proof}

In Section \ref{rfp}, we investigate the integrability of solutions to the Fokker-Planck equation. At first, we establish upper bounds for norms of $m$ in terms of norms of $D_pH(x,Du)$, as in the following Proposition:

\begin{Proposition}\label{munif}
Let $(u,m,\bar{H})$ be a classical solution of \eqref{mfg} and assume that $$\frac{1}{p}+\frac{1}{p'}=1.$$ Under the Assumptions A\ref{Hsub}-A\ref{DpH}, there exists a constant $C_p>0$, for any $p\in (1,\max\{1+\alpha,2^*/2\})$, such that:
\[
\|m\|_{L^r(\Tt^d)}\leq C_p \left(1+\left\||D_pH|^2\right\|_{L^{p'}(\Tt^d)}\right)^{\frac{1-\frac {p} {r}}{1-\frac {2p} {2^*}} },
\]
for every $r>p$. As a consequence, we have:
\[
\|m\|_{L^{\infty}(\Tt^d)}\leq C_p \left(1+\left\||D_pH|^2\right\|_{L^{p'}(\Tt^d)}\right)^{\frac{1}{1-\frac {2p} {2^*}} }
.
\]
\end{Proposition}Proposition \ref{munif} plays an instrumental role in addressing mean-field game systems in the presence of power-like nonlinearities. Further, we address the integrability of $\ln(m)$, where $m$ solves the second equation in \eqref{mfg}, in the classical sense. It is the content of the following:

\begin{Proposition}\label{logprop}
Assume that $(u,m,\bar{H})$ is a classical solution of \eqref{mfg} and let $l>1$. Suppose in addition that $p,\,q\,>1$ satisfy \eqref{conjugate} and \eqref{q2star}. Assume further that A\ref{DpH} and A\ref{ggrow}.b hold. Then, there exists $C_l>0$ so that
$$\left\|g\right\|_{L^{q(l+1)}(\Tt^d)}\,\leq\,C\,+\,C_l\left\|Du\right\|_{L^{2p(\gamma-1)}(\Tt^d)}^{\gamma-1}.$$
\end{Proposition}The proofs of Proposition \ref{munif} and \ref{logprop} are detailed in Sections \ref{rfpp} and \ref{intlog}, respectively. Section \ref{bernstein} presents a Sobolev type of estimate for solutions to the Hamilton-Jacobi equation in terms of the nonlinearity $g$; this is given as in the next Proposition:

\begin{Proposition}\label{cor1}
Suppose that $(u,m,\bar{H})$ is a solution of \eqref{mfg} and assume that A\ref{Hsub}-A\ref{gpos} hold. Then, for $p>1$ sufficiently large, we have $$\left\|Du\right\|_{L^p(\Tt^d)}\,\leq\,C_p\,+\,C_p\left\|g\right\|_{L^{r_p}(\Tt^d)},$$
where $r_p\to d$ as $p\to +\infty.$
\end{Proposition}To prove the previous Proposition we use the integral Bernstein method, see \cite{Lions85}. We follow the ideas in \cite{cirant}, where these techniques were firstly used in the context of mean-field game theory.

Our argument proceeds by combining Propositions \ref{cor1} with \ref{munif} and \ref{logprop} to explore the interplay between the regularity of the Hamilton-Jacobi equation in terms of the Fokker-Planck equation and vice-versa. The proper manipulation of this interplay yields uniform estimates. These enable us to argue through the continuation method (see, for example, \cite{GPatVrt} and \cite{GMit}) and conclude the proof of Theorem \ref{maintheorem}. Those uniform estimates are reported in what follows.

The next result regards the case of polynomial couplings (Assumption A\ref{ggrow}.a).
\begin{teo}\label{teo1}
Suppose that Assumptions A\ref{Hsub}-A\ref{ggrow}.a and  A\ref{alpha} are satisfied. Then, there exist positive constants $C$ and $C_p$, for any $p>1$, such that for any classical solution $(u,m,\bar{H})$ of \eqref{mfg}, we have:
\[
\|m\|_{L^{\infty}(\Tt^d)}\leq C,
\]and 
\[
\|Du\|_{L^p(\Tt^d)}\leq C_p.
\]
\end{teo} To address the setting of logarithmic nonlinearities, the next Theorem is instrumental.

\begin{teo}\label{teo2}
Suppose that Assumptions A\ref{Hsub}-A\ref{ggrow}.b and A\ref{gamma} hold. Then, there exist positive constants $C_a$ and $C_b$,  for $a,\,b>d$, such that any classical solution $(u,m,\bar{H})$ of \eqref{mfg} satisfies:
\[
\|g(m)\|_{L^{a}(\Tt^d)}\leq C_a,
\]and 
\[
\|Du\|_{L^b(\Tt^d)}\leq C_b.
\]
\end{teo}

The proofs of Theorem \ref{teo1} and Theorem \ref{teo2} can be found in Section \ref{proofmt}. Once these are established, we conclude the proof of Theorem \ref{maintheorem} by recurring to the continuation method. 

\section{Regularity of the distribution}\label{rfp}

Here, we investigate the integrability of solutions for the Fokker-Planck equation in \eqref{mfg}. First, we establish upper bounds for norms of $m$ in appropriate $L^p(\Tt^d)$ spaces; though interesting on their own, these estimates are natural in mean-field games with polynomial nonlinearities of the type $g(m)=m^\alpha$. 

\subsection{Integrability of solutions of the Fokker-Planck equation}\label{rfpp}

We start with a few basic estimates. Similar results have already been considered in \cite{GPatVrt} for Hamiltonians behaving as quadratic at the infinity.

\begin{Lemma}\label{barH}
Let $(u,m,\bar{H})$ be a classical solution of \eqref{mfg}. Under the Assumptions A\ref{Hsub} and A\ref{Lsup}, there exists a constant $C>0$ such that:
\[
\overline{H}\geq -C-\int_{\Tt^d}g(m)dx, 
\]
and
\[
\overline{H}+\int_{\Tt^d}|Du|^{\gamma}mdx +\int_{\Tt^d}g(m)mdx\leq C.
\]

\end{Lemma}
\begin{proof}
The first inequality is obtained by integrating the first equation in \eqref{mfg} and using Assumption A\ref{Hsub}. For the second assertion, we multiply the first equation by $m$ and subtract the second equation multiplied by $u$; integration by parts together with Assumption A\ref{Lsup} yields the result.
\end{proof}

\begin{Corollary}\label{HDu}
Under the Assumptions A\ref{Hsub}, A\ref{Lsup} and A\ref{gpos} there exists a constant $C>0$ such that for any classical solution $(u,m,\bar{H})$ of \eqref{mfg} we have:
\[
|\overline{H}|, \int_{\Tt^d}g(m)dx, \int_{\Tt^d}|Du|^{\gamma}mdx \leq C.
\]
\end{Corollary}
\begin{proof}
We observe that Lemma \ref{barH} yields uniform lower bounds for $\bar{H}$, as follows:
\begin{align}\label{dmfo1}
-C\,-\,\bar{H}\,&\leq\,\int_{\Tt^d}g(m)dx\,\leq\, g(2)\,+\,\frac{1}{2}\int_{m\geq 2}g(m)mdx\\\nonumber&\leq\,C\,-\,\frac{\bar{H}}{2}.
\end{align}
Furthermore, Assumption A\ref{gpos} implies 
\begin{equation}\label{dmfo2}
\int_{\Tt^d}g(m)dx\leq g(1)+\int_{m>1}g(m)mdx\leq C.
\end{equation} By combining \eqref{dmfo1} with \eqref{dmfo2}, one obtains the result.
\end{proof}

\begin{Corollary}\label{mlnm}
Let $(u,m,\bar{H})$ be a solution of \eqref{mfg}. Under the Assumptions A\ref{Hsub}, A\ref{Lsup} and A\ref{gpos}, we have:
\begin{enumerate}
\item $m\,\in\, L^{1+\alpha}(\Tt^d)$;
\item $\ln(m)\,\in\, L^{1}(\Tt^d)$.
\end{enumerate}
\end{Corollary}
\begin{proof}
The first assertion follows immediately from Lemma \ref{barH} and the proof of Corollary \ref{HDu}. For the second one, notice that
\[
\int_{\Tt^d}\ln(m)=\int_{m>1}ln(m)+\int_{m<1}\ln(m)\leq C+\int_{m<1}\ln(m),
\]where the inequality follows from the sublinearity of $\ln(z)$ when $z>1$. In addition, by integrating the first equation in \eqref{mfg} and using Corollary \ref{HDu}, one gets
\[
\int_{\Tt^d}\ln(m)\,\geq\,-C,
\]for some constant $C>0$. These imply
\[
-\int_{m<1}\ln(m)\,\leq\, C.
\]
Finally, we have 
\[	
\int_{\Tt^d}\left|\ln(m)\right|dx\,=\,\int_{\Tt^d}\ln(m)-2\int_{m<1}\ln(m)\leq C,
\]which concludes the proof.

\end{proof}

In the sequel, we present the proof of Proposition \ref{munif}.

\begin{proof}[Proof of Proposition \ref{munif}]
We multiply the second equation in \eqref{mfg} by $m^{r}$ for $r>0$ or by $\ln m$ for $r=0$ and integrate by parts. This yields
\begin{align}\label{rgeq0}
\int_{\Tt^d} m^{r-1}|Dm|^2dx=\int_{\Tt^d}D_pHm^rDm&\leq \frac 1 2 \int_{\Tt^d}|D_pH|^2m^{r+1}dx
\\\nonumber&\quad+\frac 1 2 \int_{\Tt^d}m^{r-1}|Dm|^2,
\end{align}for all $r\geq 0$.
For $r=0$, the inequality in \eqref{rgeq0}
gives 
\[
\int_{\Tt^d} |D\sqrt m|^2dx=\frac 1 4 \int_{\Tt^d} m^{-1}|Dm|^2dx\leq \int_{\Tt^d}|D_pH|^2mdx.
\]

In general, for $r\geq 0$, by using Holder's inequality we obtain
\[
\int_{\Tt^d} m^{r-1}|Dm|^2dx\leq \int_{\Tt^d}|D_pH|^2m^{r+1}dx\leq \left\||D_pH|^2\right\|_{p'}\left(\int_{\Tt^d}m^{p(r+1)}dx\right)^{\frac 1 p}.
\]
Since $$\int_{\Tt^d} m^{r-1}|Dm|^2dx\,=\,\frac{4}{(r+1)^2} \int_{\Tt^d} |D m^{\frac{r+1}{2}} |^2dx,$$ by combining the computations above with Sobolev's inequality, we get
\[
\left(\int_{\Tt^d}m^{\frac{2^*}{2}(r+1)}dx\right)^{\frac{2}{2^*}}\leq \frac{(r+1)^2}{4}\left(1+\left\||D_pH|^2\right\|_{p'}\right)\left(\int_{\Tt^d}m^{p(r+1)}dx\right)^{\frac 1 p}.
\]
We choose  $1\leq p< \max\{1+\alpha,2^*/2\}$ and let $\beta=\frac{2^*}{2p}$. Setting $r+1=\beta^n$ in the above estimate yields
\[
\left(\int_{\Tt^d}m^{p\beta^{n+1}}dx\right)^{\frac{2}{2^*}}\leq \frac{\beta^{2n}}{4}\left(1+\left\||D_pH|^2\right\|_{p'}\right)\left(\int_{\Tt^d}m^{p\beta^n}dx\right)^{\frac 1 p},
\]
and therefore,
\[
\|m\|_{L^{p\beta^{n+1}}(\Tt^d)}\leq 4^{-\frac {1}{\beta^n}}\beta^{\frac {2n}{\beta^n}}\left(1+\left\||D_pH|^2\right\|_{p'}\right)^{\frac {1}{\beta^n}} \|m\|_{L^{p\beta^{n}}(\Tt^d)}.
\]
Iterating, we get
\[
\|m\|_{L^{p\beta^{n}}(\Tt^d)}\leq C_p \left(1+\left\||D_pH|^2\right\|_{p'}\right)^{\frac{1-\beta^{-n}}{1-\beta^{-1}} }\|m\|_{L^p(\Tt^d)}.
\]
Because $p<1+\alpha$, the first assertion of Corollary \ref{HDu} yields $\|m\|_{L^p(\Tt^d)}\leq C$. Interpolating between $p\beta^n$ and $p\beta^{n+1}$ and using concavity of function $s\mapsto 1-1/s$, we further get 
\[
\|m\|_{L^r(\Tt^d)}\leq C_p \left(1+\left\||D_pH|^2\right\|_{p'}\right)^{\frac{1-\frac {p} {r}}{1-\beta^{-1}} },
\]for every $r>p$.
\end{proof}

In what follows, we put forward a series of estimates aimed at investigating problems with logarithmic dependence on the measure.

\subsection{Log-integrability for the Fokker-Planck equation}\label{intlog}

We open this section with a few preliminary bounds.

\begin{Lemma}\label{loglemma0}
Assume that $(u,m,\bar{H})$ is a classical solution of \eqref{mfg} and let $r\leq 2$. Then, there exists a constant $C>0$ so that
\begin{equation*}
\left\|\ln(m)\right\|_{L^r(\Tt^d)}\,\leq\,C\left\|D\ln(m)\right\|_{L^2(\Tt^d)}\,+\,C.
\end{equation*}
\end{Lemma}
\begin{proof}
The result follows from Poincar\'e's inequality:
\begin{align*}
\left\|\ln(m)\right\|_{L^r(\Tt^d)}\,\leq\,\left|\int_{\Tt^d}\ln(m)dx\right|\,+\,C\left\|D\ln(m)\right\|_{L^2(\Tt^d)}.
\end{align*}
and the Corollary \ref{mlnm}..
\end{proof}

\begin{Lemma}\label{loglemma01}
Assume that $(u,m, \overline{H})$ is a classical solution of \eqref{mfg} and let $r\leq 2$. Then, there exists a constant $C>0$, which does not depend on the solution, so that
$$\left\|\ln(m)\right\|_{L^r(\Tt^d)}\leq C\left\|D_pH\right\|_{L^2(\Tt^d)}.$$
\end{Lemma}
\begin{proof}
Multiplying the second equation in \eqref{mfg} by $\frac{1}{m}$ and integrating by parts, one obtains $$\left\|D\ln(m)\right\|_{L^2(\Tt^d)}\leq\left\|D_pH\right\|_{L^2(\Tt^d)}.$$ By recurring to Lemma \ref{loglemma0}, the proof follows.
\end{proof}

\begin{Lemma}\label{loglemma1}
Assume that $(u,m, \overline{H})$ is a classical solution of \eqref{mfg} and let $l>1$. Then, we have
\begin{equation*}
\int_{\Tt^d}\left|D\ln(m)^\frac{l+1}{2}\right|^2dx\,\leq\,\frac{(l+1)^2}{4}\int_{\Tt^d}\left|\ln(m)\right|^{l-1}\left|D_pH\right|^2dx.
\end{equation*}
\end{Lemma}
\begin{proof}
Let $F(m)$ be given by $$F(m)\doteq\int_1^m\frac{\left|\ln(y)\right|^k}{y^2}dy,$$for $k>0$. Multiply the second equation in \eqref{mfg} by $F$ and integrate by parts to obtain
\begin{equation}\label{logk}
\int_{\Tt^d}\left|\ln(m)\right|^k\left|\frac{Dm}{m}\right|^2dx\,\leq\,\int_{\Tt^d}\left|\ln(m)\right|^k\left|D_pH\right|^2.
\end{equation}On the other hand, 
\begin{equation}\label{logl}
\int_{\Tt^d}\left|D\ln(m)^\frac{l+1}{2}\right|^2dx\,=\,\frac{(l+1)^2}{4}\int_{\Tt^d}\left|\ln(m)\right|^{l-1}\left|\frac{Dm}{m}\right|^2dx.
\end{equation}By setting $k\equiv l-1$ and combining \eqref{logk} and \eqref{logl}, the result follows.
\end{proof}

In the remainder of this section, we assume $p,\,q\,>\,1$ are conjugate exponents, i.e.,
\begin{equation}\label{conjugate}
\frac{1}{p}\,+\,\frac{1}{q}\,=\,1,
\end{equation}and
\begin{equation}\label{q2star}
1\,<\,q\,<\,\frac{2^*}{2},
\end{equation}where $2^*$ is the critical Sobolev exponent, $$2^*=\frac{2d}{d-2}.$$

\begin{Corollary}\label{loglemma2}
Assume that $(u,m, \overline{H})$ is a classical solution of \eqref{mfg} and let $l>1$. Suppose in addition that $p,\,q\,>1$ satisfy \eqref{conjugate} and \eqref{q2star}. Then, there exists a constant $C_l>0$ so that:
\begin{equation*}
\left\|\ln(m)^\frac{l+1}{2}\right\|_{L^{2^*}(\Tt^d)}\leq C_l\left[\left\|\ln(m)\right\|_{L^{l+1}(\Tt^d)}^\frac{l+1}{2}+\left\|\ln(m)\right\|_{L^{q(l-1)}(\Tt^d)}^\frac{l-1}{2}\left\|D_pH\right\|_{L^{2p}(\Tt^d)}\right].
\end{equation*}
\end{Corollary}
\begin{proof}
Sobolev's inequality implies
\begin{align*}
\left\|\ln(m)^\frac{l+1}{2}\right\|_{L^{2^*}(\Tt^d)}&\leq C\left(\int_{\Tt^d}\left|\ln(m)\right|^{l+1}\right)^\frac{1}{2}+C\left(\int_{\Tt^d}\left|D\ln(m)^\frac{l+1}{2}\right|^2\right)^\frac{1}{2}.
\end{align*}Lemma \ref{loglemma1} combined with the former inequality yields
\begin{align*}
\left\|\ln(m)^\frac{l+1}{2}\right\|_{L^{2^*}(\Tt^d)}&\leq C\left(\int_{\Tt^d}\left|\ln(m)\right|^{l+1}\right)^\frac{1}{2}+C_l\left(\int_{\Tt^d}\left|\ln(m)\right|^{l-1}\left|D_pH\right|^2dx\right)^\frac{1}{2}.
\end{align*}Because of \eqref{conjugate}, an application of the H\"older's inequality concludes the proof.
\end{proof}

Next, we present the proof of Proposition \ref{logprop}.

\begin{proof}[Proof of Proposition \ref{logprop}]
H\"older inequality yields
\begin{equation}\label{holder}
\left\|\ln(m)\right\|_{L^{q(l+1)}(\Tt^d)}\,\leq\,C\left\|\ln(m)\right\|^\lambda_{L^1(\Tt^d)}\left\|\ln(m)^\frac{l+1}{2}\right\|^{\frac{2(1-\lambda)}{l+1}}_{L^{2^*}(\Tt^d)},
\end{equation}provided 
\begin{equation}\label{holdercond}
\frac{1}{q(l+1)}=\lambda+\frac{2(1-\lambda)}{2^*(l+1)}
\end{equation}is satisfied. By combining \eqref{holder} with Lemma \ref{loglemma01} and Corollary \ref{loglemma2}, we obtain the following inequality:
\begin{align*}
\left\|\ln(m)\right\|_{L^{q(l+1)}(\Tt^d)}\leq &C\left\|\ln(m)\right\|_{L^{l+1}(\Tt^d)}^{1-\lambda}+C\left\|D_pH\right\|_{L^{2p}(\Tt^d)}^\frac{2(1-\lambda)}{l+1}\left\|\ln(m)\right\|_{L^{q(l-1)}(\Tt^d)}^{\frac{(1-\lambda)(l-1)}{l+1}}.
\end{align*}Because $p,\,q>1$, successive application of the weighted Young's inequality leads to 
\begin{align*}
\left\|\ln(m)\right\|_{L^{q(l+1)}(\Tt^d)}&\leq C+C_l\left\|D_pH\right\|_{L^{2p}(\Tt^d)}.
\end{align*}Finally, Assumption A\ref{DpH} implies 
\begin{align}\label{eqlog}
\left\|\ln(m)\right\|_{L^{q(l+1)}(\Tt^d)}&\leq C+C_l\left\|Du\right\|_{L^{2p(\gamma-1)}(\Tt^d)}^{\gamma-1};
\end{align}the result follows from \eqref{eqlog} and A\ref{ggrow}.b.
\end{proof}

\section{Regularity for the Hamilton-Jacobi equation by the integral Bernstein method}\label{bernstein}

Here, we present an estimate for the Lipschitz regularity of solutions to the Hamilton-Jacobi equation in terms of norms of $g$ in appropriate Lebesgue spaces. This is based on the techniques introduced in \cite{Lions85}. We argue along the same lines as in \cite{cirant}, where the detailed proofs, in the presence of Neuman boundary data, are presented. Since the arguments in the periodic setting are very similar to those in \cite{cirant}, we omit the details in what follows.

Let $(u,m, \overline{H})$ be a solution to \eqref{mfg}. We start by setting $v\,\doteq\,\left|Du\right|^2$. Then,
\begin{align}\label{eq1}
-\Delta v \,&=\,-2\sum_{i,j=1}^d\left(D_{ij}u\right)^2\,-\,2\sum_{i=1}^dD_iuD_i\left(H(x,p)-\overline{H}-g\right).
\end{align}By multiplying \eqref{eq1} by $v^p$ and integrating over $\Tt^d$, one obtains
\begin{align}\label{eq2}
-\int_{\Tt^d}v^p\Delta v \,&+\,2\int_{\Tt^d}\left|D^2u\right|^2v^p\,=\,-\,2\int_{\Tt^d}Du\cdot D_xH(x,Du)v^p\,\\\nonumber&\quad-\,2\int_{\Tt^d}D_pH\cdot Dv\, v^p\,-\,2\int_{\Tt^d}Dg\cdot Du\,v^p.
\end{align}
The  following Lemmas establish bounds for each of the terms on the right-hand side of \eqref{eq2}.

\begin{Lemma}\label{lemma1}
Let $u\colon\Tt^d\to\Rr$ be a $C^2$ function and let $v=|Du|^2$ Then, there exist positive constants $c,\,C$, which do not depend on $u,$ such that for every $p>1$
$$-\int_{\Tt^d}v^p\Delta v\,\geq\, \frac{4pc}{(p+1)^2}\left[\left(\int_{\Tt^d}\left|v\right|^\frac{(p+1)d}{d-2}\right)^\frac{d-2}{d}-\left(\int_{\Tt^d}\left|v\right|^{p+\gamma}\right)^\frac{p+1}{p+\gamma}\right]$$
and
$$-2\int_{\Tt^d}Dg\cdot Du\,v^p\leq \frac{1}{2}\int_{\Tt^d}\left|D^2u\right|^2v^p+C\int_{\Tt^d}\left|g\right|^2v^p.$$
\end{Lemma}
\begin{proof}
See \cite[Theorem 19, page 25]{cirant}.
\end{proof}

\begin{Lemma}\label{lemma3}
Let $u\colon\Tt^d\to\Rr$ be a $C^2$ function and let $v=|Du|^2$. Assume Assumption A\ref{a2} holds. 
Then, for all $\delta>0$ there exists a constant $C_{\delta}>0$, which does not depend on $u,$ such that
 $$\left|-2\int_{\Tt^d}Du\cdot D_xH(x,Du)v^p\right|\,\leq\,C_{\delta}+\,\delta\int_{\Tt^d}v^{p+\gamma},$$ for all  $p>1.$
\end{Lemma}
\begin{proof}
It is enough to check that 
\[\left|Du\cdot D_xH(x,Du)v^p\right|\leq C(v^{p+\frac {\gamma+1}{2}}+v^{p+\frac {1}{2}})\leq \delta v^{p+\gamma}+C_{\delta}.\]
\end{proof}

\begin{Lemma}\label{lemma4}
	Let $u\colon\Tt^d\to\Rr$ be a $C^2$ function and let $v=|Du|^2$. Assume Assumption A\ref{a2} holds. 
	Then, for all $\delta>0$ there exists a constant $C_{\delta}>0$, which does not depend on $u,$ such that $$2\int_{\Tt^d}D_pH\cdot Dv\,v^p\,\leq\,\delta\int_{\Tt^d}\left|D^2u\right|^2v^p\,+\,\frac{C_{\delta}}{p+1}\int_{\Tt^d}v^{p+\gamma}\,+\,\frac{C_\delta}{p+1},$$ for every $p>1$.
\end{Lemma}
\begin{proof}
Integration by parts and Assumption A\ref{a2} yields:
\begin{align*}
&-2\int_{\Tt^d}v^pD_pH\cdot Dv\,=\frac 2 {p+1}\int_{\Tt^d}v^{p+1}\div\left(D_pH\right)\\&
\leq\,\frac C {p+1}\left[ \int_{\Tt^d}v^{p+1}\left|D^2_{xp}H\right|+ \int_{\Tt^d}v^{p+1}\left|D^2_{pp}H\right||D^2u|\right]\\&
\leq\,\frac C {p+1}\left[ \int_{\Tt^d}  (v^{p+\frac {\gamma+1}{2}}+v^{p+1})+ \int_{\Tt^d}v^{p+\frac {\gamma}{2}}|D^2u|\right]\\&
\leq\,\delta\int_{\Tt^d}\left|D^2u\right|^2v^p\,+\,\frac{C_{\delta}}{p+1}\int_{\Tt^d}v^{p+\gamma}\,+\,\frac{C_\delta}{p+1}.
\end{align*}
\end{proof}

\begin{Lemma}\label{lemma5}
Let $(u,m,\bar{H})$ be a $C^{2}$ soluition to \eqref{mfg} and $v=|Du|^2$. Suppose that Assumptions A\ref{Hsub}-A\ref{gpos} hold. Then, for any $p>1$ large enough, there exists $C_p>0$, that does not depend on $u,$ such that
\[
\left(\int_{\Tt^d}v^\frac{d(p+1)}{d-2}\right)^\frac{(d-2)}{d(p+1)}\leq C_p \left(\int_{\Tt^d}\left|g\right|^{2\beta_p} \right)^{\frac 1 {\beta_p}}+\,C_p,
\]
where $\beta_p$ is the conjugate power of $\frac{d(p+1)}{(d-2)p}$; thus $\beta_p\to \frac d 2$ when $p\to\infty.$
\end{Lemma}
\begin{proof}
Combining the estimates from Lemmas \ref{lemma1}, \ref{lemma3} and \ref{lemma4} one obtains
\begin{align}\label{ineq1}
C_p&\left(\int_{\Tt^d}v^\frac{d(p+1)}{d-2}\right)^\frac{d-2}{d}+\left[2-\left(\frac{1}{2}+\delta\right)\right]\int_{\Tt^d}\left|D^2u\right|^2v^p\leq\\\nonumber&C_p\left(\int_{\Tt^d}v^{p+\gamma}\right)^\frac{p+1}{p+\gamma}+C\int_{\Tt^d}\left|g\right|^2v^p+\left(\frac {C_\delta} {p+1}+\delta\right)\int_{\Tt^d}v^{p+\gamma}+C_{\delta}.
\end{align}

Using the first equation in \eqref{mfg}, Assumption A\ref{Hsub}, and the bounds on $\overline{H}$ from Corollary \ref{HDu}, we have
\begin{align}\label{ineq2}
&\int_{\Tt^d}\left|D^2u\right|^2v^p\geq \frac{1}{d}\int_{\Tt^d}\left|\Delta u\right|^2v^p=\frac{1}{d}\int_{\Tt^d}\left|H-g-\overline{H} \right|^2v^p\\\nonumber&\geq \frac{1}{3d}\int_{\Tt^d}H^2v^p -\frac{1}{d}\int_{\Tt^d}g^2v^p-\frac{1}{d}C \int v^p\\\nonumber&
\geq c\int_{\Tt^d}v^{p+\gamma}-C \int_{\Tt^d}g^2v^p-C,
\end{align}where the second inequality follows from $(a-b-c)^2\geq \frac 1 3 a^2 -b^2 -c^2.$

For small values of $\delta$ and $p$ large enough, inequalities \eqref{ineq1} and \eqref{ineq2} yield
\begin{align*}
\left(\int_{\Tt^d}v^\frac{d(p+1)}{d-2}\right)^\frac{d-2}{d}&\leq C_p\int_{\Tt^d}\left|g\right|^2v^p\quad+\,C_p\\&\leq
 C_p\left(\int_{\Tt^d}v^\frac{d(p+1)}{d-2}\right)^\frac{(d-2)p}{d(p+1)} \left(\int_{\Tt^d}\left|g\right|^{2\beta_p} \right)^{\frac 1 {\beta_p}}+\,C_p.
\end{align*}Hence,
\[
\left(\int_{\Tt^d}v^\frac{d(p+1)}{d-2}\right)^\frac{(d-2)}{d(p+1)}\leq C_p \left(\int_{\Tt^d}\left|g\right|^{2\beta_p} \right)^{\frac 1 {\beta_p}}+\,C_p.
\]
\end{proof}

We close this section with the proof of Proposition \ref{cor1}.

\begin{proof}[Proof of Proposition \ref{cor1}] The result follows easily from Lemma \ref{lemma5}.
\end{proof}

\section{Proofs of Theorem \ref{teo1} and Theorem \ref{teo2}}\label{proofmt}

In the sequel, we detail the proof of Theorem \ref{teo1}. This is done by combining the results in Proposition \ref{munif} with Proposition \ref{cor1}.

\begin{proof}[Proof of Theorem \ref{teo1}]
Let $p\in(1,1+\alpha)$ and notice that $q\doteq 2(\gamma-1)p'\to\infty$ as $p\to 1$. Proposition \ref{cor1} leads to
\[
\left\|Du\right\|_{L^{2(\gamma-1)p'}(\Tt^d)}\,\leq\,C_q\,+\,C_q\left\|g\right\|_{L^{r_q}(\Tt^d)}\leq C_q\,+\,C_q\left\|m\right\|_{L^{\alpha r_q}(\Tt^d)},
\]
where we also have used Assumption A\ref{ggrow}. On the other hand, from Lemma \ref{munif} and A\ref{DpH}
\[
\|m\|_{L^{\alpha r_q}(\Tt^d)}\leq C_p \left(1+\left\|Du\right\|_{L^{2(\gamma-1)p'}(\Tt^d)}\right)^{\frac{2(\gamma-1)(1-\frac {p} {\alpha r_q})}{1-\frac {2p} {2^*}} }.
\]
Combining both inequalities, we get
\[
\left\|Du\right\|_{L^{2(\gamma-1)p'}(\Tt^d)}\,\leq C_p \left(1+\left\|Du\right\|_{L^{2(\gamma-1)p'}(\Tt^d)}\right)^{\frac{\alpha(\gamma-1)(1-\frac {p} {\alpha r_q})}{1-\frac {2p} {2^*}} },\]
for all $r>d$ and $p\in (1,1+\alpha)$.
By taking $p\to 1$ we get $q\to\infty$ and $r_q\to d$; hence the exponent in the previous inequality is such that
$$\frac{2\alpha(\gamma-1)\left(1-\frac{p}{\alpha r_q}\right)}{1-\frac{2p}{2^*}}\to d(\gamma-1)\alpha-(\gamma-1),$$
which is strictly less than one because of Assumption A\ref{alpha}. We then conclude that $
\|Du\|_{L^q(\Tt^d)}\leq C_q,
$ for large enough $q.$
Lemma \ref{munif} and Assumption A\ref{DpH} then yield
\[
\|m\|_{L^{\infty}(\Tt^d)}\leq C_q \left(1+\left\||D_pH|^2\right\|_{L^{q}(\Tt^d)}\right)^{\frac{1}{1-\frac {2q'} {2^*}} }\leq C_q
\]
for any $q>1.$
\end{proof}

Next, we combine the estimates in Proposition \ref{logprop} with the a priori regularity for the Hamilton-Jacobi equation obtained in Proposition \ref{cor1} to establish Theorem \ref{teo2}.

\begin{proof}[Proof of Theorem \ref{teo2}]
Lemma \ref{lemma5} yields
\begin{equation*}
\left\|Du\right\|_{L^\frac{2d(r+1)}{d-2}(\Tt^d)}\,\leq\,C_r\,+\,C_r\left\|g(m)\right\|_{L^\frac{2d(r+1)}{d+2r}(\Tt^d)}.
\end{equation*}
On the other hand, Proposition \ref{logprop} implies
\[
\left\|g\right\|_{L^{q(l+1)}(\Tt^d)}\leq C+C_l\left\|Du\right\|^{\gamma-1}_{L^{2p(\gamma-1)}(\Tt^d)}.
\]By choosing $r$ large enough and noticing that $$\frac{2d(r+1)}{d+2r}<d,$$ we have
\[
\left\|g\right\|_{L^{q(l+1)}(\Tt^d)}\leq C+C_l\left\|Du\right\|_{L^\frac{2d(r+1)}{d-2}(\Tt^d)}^{\gamma-1}\,\leq\,C_r\,+\,C_r\left\|g(m)\right\|_{L^d(\Tt^d)}^{\gamma-1}.
\]
Since we can always choose $l$ so that $d<q(l+1)$, we get
\[
\left\|g\right\|_{L^{d}(\Tt^d)}\leq C\left\|g\right\|_{L^{1}(\Tt^d)}^\theta\left\|g\right\|_{L^{q(l+1)}(\Tt^d)}^{1-\theta},
\]for $\theta$ determined by 
\[
\frac{1}{d}=\theta+\frac{1-\theta}{q(l+1)}.
\]

By gathering these computations and using Lemma \ref{mlnm}, we obtain:
\[
\left\|g\right\|_{L^{q(l+1)}(\Tt^d)}\leq C_l\left(1+\left\|g\right\|_{L^{q(l+1)}(\Tt^d)}^{(\gamma-1)(1-\theta)}\right).
\]Assumption A\ref{gamma} implies that $(\gamma-1)(1-\theta)<1$, for large enough $l$, which then leads to 
\[
\left\|g\right\|_{L^a(\Tt^d)}\leq C_a,
\]for $a>d$.
By using this bound in Lemma \ref{lemma5}, we conclude that
\[
\left\|Du\right\|_{L^b(\Tt^d)}\leq C_b.
\]
\end{proof}

Additional arguments build upon Theorems \ref{teo1} and \ref{teo2} to yield improved regularity for $u$ and $m$.

\begin{Lemma}\label{mbelow}
Assume Assumptions A\ref{Hsub}-A\ref{gpos} hold. Suppose that either Assumptions A\ref{ggrow}.a and A\ref{alpha} or A\ref{ggrow} and A\ref{gamma}.b are satisfied. Then, there exists a constant $C>0,$ such that for any $C^{\infty}$ solution $(u, m, \overline{H})$ to \eqref{mfg} we have $\left\|\ln m\right\|_{L^{\infty}(\Tt^d)}, \left\|D\ln m\right\|_{L^{\infty}(\Tt^d)}\leq C.$ In particular, there exists $m_0>0,$ so that $m\geq m_0.$
\end{Lemma}
\begin{proof}
Using elliptic regularity in the first equation of \eqref{mfg} and Theorems \ref{teo1} and \ref{teo2} we conclude that $\|u\|_{W^{2,q}(\Tt^d)}\leq C_q,$ for every $q>1.$ In particular, by Morrey embedding Theorem, we have also $L^{\infty}$ bounds for $u.$  Further, note that the variable $w=\ln m$ satisfies the following equation
\[
\Delta w +|Dw|^2-D_pH(x,Du)\cdot Dw-\div(D_pH(x,Du))=0.
\]
The result follows from the above-stated regularity of $Du$ and nonlinear adjoint techniques (cf. Proposition 6.9 \cite{GPatVrt}.)
\end{proof}
	
\begin{teo}\label{teom}
Suppose Assumptions A\ref{Hsub}-A\ref{gpos} hold. Moreover, assume that either Assumptions A\ref{ggrow}.a and A\ref{alpha} or A\ref{ggrow} and A\ref{gamma}.b are satisfied. Then, for any $k\geq 1,q>1,$ there exist a constant $C_{k,q}>0$, such that for any $C^{\infty}$ solution $(u, m, \overline{H})$ to \eqref{mfg} we have $$\left\|D^ku\right\|_{L^q(\Tt^d)}, \left\|D^km\right\|_{L^q(\Tt^d)}\leq C_{k,q}.$$
\end{teo}

\begin{proof}
	The proof follows from the Theorems \ref{teo1},\ref{teo2} and Lemma \ref{mbelow} and standard bootstrap arguments.
\end{proof}

\appendix
\section{Existence by continuation method}
\label{exist}

To prove the existence of smooth solutions of \eqref{mfg}, we apply the continuation method, as in \cite{GPatVrt}. Since the proof there follows the same lines, here we only sketch the proof for our setting.

We consider a parametrized family of Hamiltonians:
\[
H_{\lambda}(x,p)=\lambda H(x,p)+(1-\lambda)(1+|p|^2)^{\gamma/2},\quad 0\leq \lambda\leq 1,
\]
and the corresponding system of PDE's:
\begin{equation}
	\label{mainl}
	\begin{cases}
		\Delta m_{\lambda}-\div(D_pH_{\lambda}(x,Du_{\lambda})m_{\lambda})=0\\
		\Delta u_{\lambda}+H_{\lambda}(x,Du_{\lambda})=\overline{H}_{\lambda}+g(m_{\lambda})\\
		\int_{\mathbb{T}^d}u_{\lambda}=0\\
		\int_{\mathbb{T}^d}m_{\lambda}=1.
	\end{cases}
\end{equation}
Note that for $\lambda =1$ this is exactly \eqref{mfg}.
We introduce the notation used in \cite{GPatVrt}.
\[
\dot{H}^k(\mathbb{T}^d,\Rr)=\left\{\,f\in H^k(\mathbb{T}^d,\Rr)|\int_{\mathbb{T}^d}f=0\,\right\}.
\]
Consider the Hilbert space $F^k=\dot{H}^k(\mathbb{T}^d,\Rr)\times H^k(\mathbb{T}^d,\Rr)\times L^2(\mathbb{T}^d,\Rr^d)\times\Rr$ with the norm
\[
\|w\|^2_{F^k}=\|\psi\|^2_{\dot{H}^k(\mathbb{T}^d,\Rr)}+\|f\|^2_{H^k(\mathbb{T}^d,\Rr)}+\|W\|^2_{L^2(\mathbb{T}^d,\Rr^d)}+|h|^2,
\]
for $w=(\psi,f,W,h)\in F^k.$  For $k$ large enough Sobolev's embedding theorem allows one to define the space $H^k_+(\mathbb{T}^d,\Rr)$ of positive functions in $H^k(\mathbb{T}^d,\Rr).$

Let
\[
F^k_+=\dot{H}^k(\mathbb{T}^d,\Rr)\times H^k_+(\mathbb{T}^d,\Rr)\times\Rr,
\]
by a classical solution to \eqref{mainl} we mean a tuple
$(u_{\lambda},m_{\lambda}, \overline{H}_{\lambda})\in\bigcap\limits_k F^k_+$.

\begin{proof}[Proof of Theorem \ref{maintheorem}]
	For big enough $k$ we can define
	$
	E\colon\Rr\times F^k_+\to F^{k-2}
	$
	by
	\[
	E(\lambda,u,m,\overline{H})=\left(
	\begin{array}{c}
	\Delta m-\div(D_pH_{\lambda}(x,Du)m)\\
	-\Delta u-H_{\lambda}(x,Du)+\overline{H}+g(m)\\
	-\int_{\mathbb{T}^d}m+1
	\end{array}
	\right).
	\]
	Then, \eqref{mainl} can be written as $E(\lambda,v_{\lambda})=0,$ where $v_{\lambda}=(u_{\lambda}, m_{\lambda}, \overline{H}_{\lambda}).$
	The partial derivative of $E$ in the second variable at the point $v_{\lambda}=(u_{\lambda}, m_{\lambda}, \overline{H}_{\lambda})$
	\[
	\mathcal{L}_{\lambda}:=D_2E(\lambda,v_{\lambda})\colon F^k\to F^{k-2},
	\]
	is given by
	\[
	\mathcal{L}_{\lambda}(w)(x)=\left(
	\begin{array}{c}
	\Delta f(x)-\div(D_pH_{\lambda}(x,Du_{\lambda})f(x)+m_{\lambda}D^2_{pp}H(x,Du_{\lambda}) D\psi)\\
	-\Delta \psi(x)-D_{p}H(x,Du_{\lambda}) D\psi+g'(m_{\lambda}(x))f(x)+h\\
	-\int_{\mathbb{T}^d}f
	\end{array}
	\right),
	\]
	where $w=(\psi,f,W,h)\in F^k$. Note that $\mathcal{L}_{\lambda}$ is well defined for any $k>1.$
	
	We define the set
	\[
	\Lambda=\{\,\lambda|\quad 0\leq\lambda\leq 1,\,\eqref{mainl} \text{ has a classical solution } (u_{\lambda},m_{\lambda}, \Hh_{\lambda})\,\}.
	\]
	
	Note that $0\in\Lambda$, as $(u_0,m_0,\Hh_0)\equiv(0,1,-g(1))$ is a solution to \eqref{mainl} for $\lambda=0.$ Our purpose is to prove $\Lambda=[0,1].$
	The bounds from Theorem \ref{teom} build upon the Sobolev's embedding Theorem, and Arzel\`a-Ascoli Theorem, to yield that $\Lambda$ is a closed set.
	To prove that $\Lambda$ is open, we need to prove that $\mathcal{L}_{\lambda}$ is invertible in order to use an implicit function theorem. For this, let $F=F^1$. For $w_1,w_2 \in F$ with smooth components, we can define
	\[
	B_{\lambda}[w_1,w_2]=\int_{\mathbb{T}^d} w_2\cdot \mathcal{L}_{\lambda}(w_1).
	\]
	Using integration by parts, we have for $w_1,w_2$ smooth,
	\begin{equation}
		\begin{split}
			B_{\lambda}[w_1,w_2]=\int_{\mathbb{T}^d}
			[&m_{\lambda}D\psi_1\cdot D^2_{pp}H \cdot D\psi_2+f_1D_pH_{\lambda}D\psi_2-f_2D_pH_{\lambda}D\psi_1\\&+g'(m_{\lambda})f_1f_2+D\psi_1Df_2-Df_1D\psi_2
			+h_1f_2 -h_2f_1].
		\end{split}
	\end{equation}
	This last expression is well defined on $F\times F.$ Thus it defines a bilinear form $B_{\lambda}\colon F\times F\to \Rr$.
	\begin{clm}
		$B$ is bounded $|B_{\lambda}[w_1,w_2]|\leq C\|w_1\|_F\|w_2\|_F$.
	\end{clm}
	We use the Holder's inequality on each summand.
	\begin{clm}
		There exists a linear bounded mapping $A\colon F\to F$ such that $B_{\lambda}[w_1,w_2]=(Aw_1,w_2)_F$.
	\end{clm}
	
	This follows from Claim 1 and the Riesz Representation Theorem.
	
	\begin{clm}
		There exists a positive constant $c$ such that $\|Aw\|_F\geq c\|w\|_F$ for all $w\in F.$
	\end{clm}
	If the previous claim were false there would exist a sequence $w_n\in F$ with $\|w_n\|_F=1$ such that $Aw_n\to 0.$ Let $w_n=(\psi_n,f_n, h_n)$ 
	\[
	\int_{\mathbb{T}^d} m_{\lambda}D\psi_n\cdot D^2_{pp}H_{\lambda} \cdot D\psi_n+g'(m_{\lambda})f_n^2=B_{\lambda}[w_n,w_n]\to 0.
	\]
	By Lemma \ref{mbelow} and strict convexity of $H$, $D^2_{pp}H_{\lambda}\geq\theta I>0,\text{ and } g'(m_{\lambda})>\delta>0$ where $\delta, \theta$ do not depend on the solution $v_{\lambda}$. Thus, the above expression implies that $\psi_n\to 0$ in $\dot{H}^1_0$ and $f_n\to 0$ in $L^2$. Taking $\check{w}_n=(f_n-\int f_n,0,0)\in F$ we get
	\begin{equation*}
		\int_{\mathbb{T}^d}[-|Df_n|^2+D\psi_n\cdot D^2_{pp}H \cdot Df_n+f_nD_pH_{\lambda}Df_n]=B[w_n,\check{w}_n]
		=(Aw_n,\check{w}_n),
	\end{equation*}
	\[
	\frac{1}{2}\|Df_n\|^2_{L^2(\mathbb{T}^d)}
	-C\left(\|D\psi_n\|^2_{L^2(\mathbb{T}^d)}+\|f_n\|^2_{L^2(\mathbb{T}^d)}\right)\leq-(Aw_n,\check{w}_n)\to 0,
	\]
	where $C$ depends only on $u_{\lambda}$ and $H_{\lambda}$, thus since $D\psi_n,f_n\to 0$ in $L^2$ we get that $f_n \to 0$ in $H^1(\mathbb{T}^d).$
	Now taking $\breve{w}=(0,1,0)$ we get
	\[
	\int_{\mathbb{T}^d}[-D_pH_{\lambda}D\psi_n+g'(m_{\lambda})f_n]+h_n=B[w_n,\breve{w}]
	=(Aw_n,\breve{w})\to 0,
	\]
	using the fact that $D\psi_n,f_n\to 0$ in $L^2$ we get $h_n\to 0$. We conclude that $\|w_n\|_F\to 0$, which contradicts with $\|w_n\|_F=1$.
	\begin{clm}
		$R(A)$ is closed in $F$.
	\end{clm}
	If $Au_n\to w$ in $F$ then by Claim 3 $c\|u_n-u_m\|_F\leq \|Au_n-Au_m\|_F\to 0$ as $n,m\to \infty$. Therefore
	$u_n$ converges to some $u\in F$, then $Au=w$ this proves that $R(A)$ is closed.
	\begin{clm}
		$R(A)=F$.\end{clm}
	Suppose $R(A)\neq F$, then since $R(A)$ is closed in $F$ there exists $w\neq 0$ such that $w\bot R(A)$ in $F$. Let $w=(\psi,f,h)$ then
	\[
	0=(Aw,w)=B_{\lambda}[w,w]\geq
	\int_{\mathbb{T}^d}\theta |D\psi|^2+ \delta |f|^2
	\]
	thus $\psi=0,f=0$.
	Choosing now $\bar{w}=(0,1,0)$  gives
	$h=B_{\lambda}[\bar{w},w]=(A\bar w, w)=0$.
	Thus $w=0$, and this implies $R(A)=F$.

	\begin{clm}
		For any $w_0\in F^0$ there exists a unique $w\in F$ such that $B_{\lambda}[w,\tilde{w}]=
		(w_0,\tilde{w})_{F^0}$ for all $\tilde{w}\in F.$ This implies that $w$ is the unique weak solution to the equation $\mathcal{L}_{\lambda}(w)=w_0$.
		Then, regularity theory implies that $w\in F^2$ and
		$\mathcal{L}_{\lambda}(w)=w_0$ in the sense of $F^2.$
	\end{clm}
	
	Consider the functional $\tilde{w}\mapsto(w_0,\tilde{w})_{F^0}$ on $F$. By the Riesz representation theorem, there exists $\omega\in F$ such that $(w_0,\tilde{w})_{F^0}=(\omega,\tilde{w})_F$ now taking $w=A^{-1}\omega$ we get $$B[w,\tilde{w}]=(Aw,\tilde{w})_F=(\omega,\tilde{w})_F=(w_0,\tilde{w})_{F^0}.$$
	%
	%
	Therefore $f$ is a weak solution to
	\[
	\Delta f-\div(m_{\lambda}D^2_{pp}H \cdot D\psi+fD_pH_{\lambda})=\psi_0
	\]
	and $\psi$ is a weak solution to
	
	\[
	-\Delta \psi-D_pH_{\lambda}D\psi+g'(m_{\lambda})f+h=f_0.
	\]
	Using results from the regularity theory for elliptic equations and bootstrap arguments, we conclude that $w=(\psi,f,h)\in F^2$, and thus $\mathcal{L}_{\lambda}(w)=w_0.$

	Consequently, $\mathcal{L}_{\lambda}$ is a bijective operator from $F^2$ to $F^0$. Then, $\mathcal{L}_{\lambda}$ is  injective as an operator from $F^k$ to $F^{k-2}$ for any $k\geq 2.$ To prove that it is also surjective take any $w_0\in F^{k-2}$; then, there exists $w\in F^2$ such that $\mathcal{L}_{\lambda}(w)=w_0$. Using regularity theory for elliptic equations and bootstrap arguments, we conclude that in fact $w\in F^k$. This proves that $\mathcal{L}_{\lambda}\colon F^k\to F^{k-2}$ is surjective and, therefore, also bijective.
	
	\begin{clm}
		$\mathcal{L}_{\lambda}$ is an isomorphism from $F^k$ to $F^{k-2}$ for any $k\geq 2.$
	\end{clm}
	Since we have $\mathcal{L}\colon F^k \to F^{k-2}$ is bijective we just need to prove that it is also bounded. But that follows directly from the 
	
	\begin{clm}
		We now prove that the set $\Lambda$ is open.
	\end{clm}
	We choose $k>d/2$ so that $H^k(\mathbb{T}^d,\Rr)$ is an algebra.
	For a point $\lambda_0\in \Lambda$ we have proven that the partial derivative $\mathcal{L}=D_2E(\lambda_0,v_{\lambda_0})\colon F^k \to F^{k-2}$ is an isometry. By the implicit function theorem
	(see \cite{D1}) there exists a unique solution $v_{\lambda}\in F^k_+$ to $E(\lambda,v_{\lambda})=0$ for some neighborhood $U$ of $\lambda_0$. Since  $H^k(\mathbb{T}^d,\Rr)$ is an algebra using bootstrap argument it is easy to see that $u_{\lambda}, m_{\lambda}$  are smooth, therefore $v_{\lambda}$  is a classical solution to \eqref{mfg}. Thus we conclude that $U\subset\Lambda$, which proves that $\Lambda$ is open. We have proven that $\Lambda$ is both open and closed, hence $\Lambda=[0,1]$.
\end{proof}

\bibliography{mfg}
\bibliographystyle{plain}

\end{document}